\documentclass{article}[11pt]

\usepackage{amsmath}   
\usepackage{amssymb}
\usepackage{accents}
\usepackage{amsthm}
\usepackage{tikz-cd} 
\makeatletter
\newenvironment{proofof}[1]{\par
\pushQED{\qed}%
\normalfont \topsep6\p@\@plus6\p@\relax
\trivlist
\item[\hskip\labelsep
\em Proof of #1\@addpunct{.}]\ignorespaces
}{%
\popQED\endtrivlist\@endpefalse
}
\makeatother

\newtheorem{Th}{Theorem}[section] 
   
\newtheorem{Lem}{Lemma}[section]   
   
\newtheorem{Def}{Definition}  
\newtheorem{Rem}{Remark}[section]

\newcommand{\R}{\mathbb{R}}
\newcommand{\Z}{\mathbb{Z}}
\newcommand{\C}{\mathbb{C}}
\newcommand{\T}{\mathbb{T}}

\newcommand{\h}{\mathfrak{h}}
\newcommand{\HH}{\mathcal{H}}
\newcommand{\LL}{\mathcal{L}}
\newcommand{\id}{\text{\rm id}}

\newcommand{\Span}{\mathop{\rm span}\nolimits}

\newcommand{\tr}{\text{\rm tr}\,}

\begin{document}

\title{On the non-existence of local Birkhoff coordinates for the focusing NLS equation}   
 
\author{T. Kappeler\thanks{T.K. is partially supported by the Swiss NSF.}, 
P. Topalov\thanks{P.T. is partially supported by the Simons Foundation, Award \#526907.}} 

\date{}
\maketitle

\begin{abstract}  
We prove that there exist potentials so that near them the focusing non-linear  Schr\"odinger equation does not admit 
local Birkhoff coordinates. The proof is based on the construction of a local normal form of the linearization of the equation 
at such potentials.
\end{abstract}


\section{Introduction}\label{sec:introduction}
It is well known that the non-linear Schr\"odinger (NLS) equation 
\begin{equation}\label{eq:NLS_introduction}
\left\{
\begin{array}{l}
\dot{\varphi}_1=i\varphi_{1xx}-2i\varphi_1^2\varphi_2,\\
\dot{\varphi}_2=-i\varphi_{2xx}+2i\varphi_1\varphi_2^2,
\end{array}
\right.
\end{equation}
on the torus $\T\equiv\R/\Z$ is a Hamiltonian PDE on the scale of Sobolev spaces
$H^s_c=H^s_\C\times H^s_C$, $s\ge 0$, with Poisson bracket 
\begin{equation}\label{eq:poisson_bracket_introduction}
\{F,G\}(\varphi):=
-i\int_0^1\big((\partial_{\varphi_1}F)(\partial_{\varphi_2}G)-(\partial_{\varphi_1}G)(\partial_{\varphi_2}F)\big)\,dx,
\end{equation}
and Hamiltonian $\HH : H^1_c\to\C$, given by
\begin{equation}\label{eq:NLS-Hamiltonian_introduction}
\HH(\varphi)=\int_0^1\big(\varphi_{1x}\varphi_{2x}+\varphi_1^2\varphi_2^2\big)\,dx,
\quad\varphi=(\varphi_1,\varphi_2)\in H^1_c.
\end{equation}
Here, for any $s\ge 0$, $H^s_\C\equiv H^s(\T,\C)$ denotes the Sobolev space of complex valued functions on $\T$
and the Poisson bracket \eqref{eq:poisson_bracket_introduction} is defined for functionals $F$ and $G$ on $H^s_c$,
provided that the pairing given by the integral in \eqref{eq:poisson_bracket_introduction} is well-defined 
(cf. Section \ref{sec:set-up} for more details on these matters).
The NLS phase space $H^s_c$ is a direct sum of two reals subspaces $H^s_\C=H^s_r\oplus_\R i H^s_r$ where
\begin{equation*}
H^s_r=\big\{\varphi\in H^s_c\,\big|\,\varphi_2=\overline{\varphi_1}\big\}\quad\text{and}\quad
i H^s_r=\big\{\varphi\in H^s_c\,\big|\,\varphi_2=-\overline{\varphi_1}\big\}.
\end{equation*}
The Hamiltonian vector field, corresponding to \eqref{eq:poisson_bracket_introduction} and \eqref{eq:NLS-Hamiltonian},
\[
X_\HH(\varphi)=i\big(-\partial_{\varphi_2}\HH,\partial_{\varphi_1}\HH\big)=
\big(i\varphi_{1xx}-2i\varphi_1^2\varphi_2, -i\varphi_{2xx}+2i\varphi_1\varphi_2^2\big)
\]
is tangent to the real subspaces $H^s_r$ and $i H^s_r$ (cf. Section \ref{sec:set-up}) and for any $s\ge 0$ the restrictions 
\[
X_\HH\big|_{H^2_r} : H^2_r\to L^2_r\quad\text{and}\quad X_\HH\big|_{i H^2_r} : i H^2_r\to i L^2_r
\]
are real analytic maps (cf. Section \ref{sec:set-up}). The vector field $X_\HH\big|_{H^2_r}$ corresponds to
the {\em defocusing} NLS (dNLS) equation
\[
i u_t=-u_{xx}+2|u|^2u
\]
whereas $X_\HH\big|_{i H^2_r}$ corresponds to the {\em focusing} NLS (fNLS) equation
\[
i u_t=-u_{xx}-2|u|^2u.
\]
Both equations are known to be well-posed on the Sobolev space $H^s_r$ and respectively $i H^s_r$ for any $s\ge 0$.
Moreover, they are integrable PDEs: the dNLS equation can be brought into Birkhoff normal form on the {\em entire}
phase space $L^2_r$ (cf. \cite{GK1}) whereas for the fNLS equation, an Arnold-Liouville type theorem 
has been established in \cite{KT}. The aim of this paper is to study the local properties of the vector field $X_\HH$ in
small neighborhoods of the {\em constant potentials} 
\[
\varphi_c(x)=(c,-\bar{c})\in i L^2_r\cap C^\infty_r,\quad c\in\C\setminus\{0\}.
\]
More specifically, for a given $c\in\C$, $c\ne 0$, consider the {\em re-normalized} NLS Hamiltonian
\[
\HH^c=\HH-2|c|^2\HH_1
\]
where 
\[
\HH_1(\varphi)=-\int_0^1\varphi_1(x)\varphi_2(x)\,dx.
\]
One of our main results is the following instance of an infinite dimensional version of Williamson classification theorem 
in finite dimensions (\cite{Wi}). 

\begin{Th}\label{th:hessian_normal_form_introduction}
Assume that $c\in\C$ and $|c|\notin \pi\Z$. Then there exists a Darboux basis $\big\{\alpha_k,\beta_k\big\}_{k\in\Z}$ in
$i L^2_r$ such that the Hessian $d^2_{\varphi_c}\HH^c$, when viewed as a quadratic form represented in this basis, 
takes the form
\begin{eqnarray}
d^2_{\varphi_c}\HH^c&=&4|c|^2dp_0^2-
\sum_{0<\pi k<|c|}4\pi k\sqrt[+]{|c|^2-\pi^2k^2}\big(dp_kdq_k+dp_{-k}dq_{-k}\big)\nonumber\\
&-&\sum_{\pi|k|>|c|}4\pi|k|\sqrt[+]{\pi^2k^2-|c|^2}\big(dp_k^2+dq_k^2\big)
\end{eqnarray}
where $\big\{(dp_k,dq_k)\big\}_{k\in\Z}$ are the dual coordinates.
\end{Th}

We conjecture in Section \ref{sec:linearization} that Theorem \ref{th:hessian_normal_form_introduction} can be 
generalized to a small neighborhood of the constant potential $\varphi_c$. We refer to the end of Section \ref{sec:linearization} 
where the precise statement of such a generalization is given. An analog of Theorem \ref{th:hessian_normal_form_introduction}, 
formulated in terms of the linearization of the Hamiltonian vector field $X_{\HH^c}$ at the constant potential $\varphi_c$, 
is formulated in Section \ref{sec:linearization} (see Theorem \ref{th:normal_form}). As a consequence of these results, we obtain

\begin{Th}\label{th:main_introduction}
For any given $c\in\C$ with $|c|\notin\pi\Z$ and $|c|>\pi$ the focusing NLS equation does {\em not} allow gauge invariant 
local Birkhoff coordinates in a neighborhood of the constant potential $\varphi_c$.
\end{Th}

We refer to Section \ref{sec:birkhoff_coordinates} for the precise definition of gauge invariant local Birkhoff coordinates.
In more general terms, Theorem \ref{th:main_introduction} means that there is no neighborhood of the constant potential 
$\varphi_c$ with $|c|\notin\pi\Z$ and $|c|>\pi$ where one can introduce action-angle coordinates for the fNLS equation
so that the action variables commute with the Hamiltonian $\HH_1$. Note that a similar result could be obtained using
the B\"acklund transform and the existence of a homoclinic orbit in a neighborhood of the constant potential
$\varphi_c$ (cf. \cite{LM}). Note however, that such a neighborhood of $\varphi_c$ is {\em not} arbitrarily small since 
it contains the homoclinic solution of the fNLS equation.

\medskip

Finally, note that the same results hold for the potentials
\[
\varphi_{c,k}:=\big(c e^{2\pi i k x},-\bar{c} e^{-2\pi i k x}\big)
\]
where $k\in\Z$ and $c\in\C$ with $|c|\notin\pi\Z$ and $|c|>\pi$. The only difference is that the re-normalized Hamiltonian
for these potentials is of the form $\HH_{c,k}=\HH+\alpha\HH_1+\beta\HH_2$ where $\alpha,\beta\in\C$ are specifically chosen
constants depending on the choice of $c$ and $k$ and $\HH_2(\varphi)=i\int_0^1\varphi_1(x)\varphi_{2x}(x)\,dx$.
This easily follows from the fact that for any given $k\in\Z$ and for any $s\ge 0$ the transformation
\[
\tau_k : i H^s_r\to i H^s_r,\quad\big(\varphi_1,\varphi_2\big)\mapsto\big(\varphi_1e^{2\pi i k x},\varphi_2e^{-2\pi i k x}\big),
\]
preserves the symplectic structure induced by \eqref{eq:poisson_bracket_introduction} and transforms the Hamiltonian $\HH$ into
a linear sum of the Hamiltonians $\HH$, $\HH_1$, and $\HH_2$.

\medskip\medskip

\noindent{\em Organization of the paper:} The paper is organized as follows: In Section \ref{sec:set-up} we introduce the
basic notions related to the symplectic phase geometry of the NLS equation that are needed in this paper. 
Theorem \ref{th:hessian_normal_form_introduction} and related results are proven in Section \ref{sec:linearization}.
In Section \ref{sec:birkhoff_coordinates} we prove Theorem \ref{th:main_introduction}.

\medskip\medskip

\noindent{\em Acknowledgment:} 
Both authors would like to thank the Fields Institute and its staff for the excellent working conditions, and the organizers of
the workshop ``Inverse Scattering and Dispersive PDEs in One Space Dimension" for their efforts to make it such a successful 
and pleasant event. The second author acknowledges the partial support of the Institute of Mathematics and Informatics of the 
Bulgarian Academy of Sciences.

\section{Set-up}\label{sec:set-up}
\noindent {\em 1) The NLS phase space.}
It is well-known that the non-linear Schr\"odinger equation is a Hamiltonian system on the phase space
$L^2_c:=L^2_{\C}\times L^2_{\C}$ where $L^2_{\C}\equiv L^2(\T,\C)$ is the space of square summable
complex-valued functions on the torus $\T$. For any two elements $f,g\in L^2_c$,
the Hilbert scalar product on $L^2_c$ is defined as $(f,g)_{L^2}:=\int_0^1(f_1\overline{g_1}+f_2\overline{g_2})\,dx$
where $f=(f_1,f_2)$, $g=(g_1,g_2)$, and $\overline{g_1}$ and $\overline{g_2}$ denote the complex conjugates of 
$g_1$ and $g_2$ respectively. In addition to the scalar product we will also need the non-degenerate pairing 
\begin{equation}\label{eq:pairing}
\langle f,g\rangle_{L^2}:=\int_0^1(f_1g_1+f_2g_2)\,dx\,.
\end{equation}
The symplectic structure on $L^2_c$ is
\begin{equation}\label{eq:symplectic_structure}
\omega(f,g):=-i \int_0^1
\det
\begin{pmatrix}
f_1&g_1\\
f_2&g_2
\end{pmatrix}
\,dx
\end{equation}
(Note that $\omega(f,g)$ is {\em not} the K\"ahler form of the Hermitian scalar product $(\cdot,\cdot)_{L^2}$ in $L^2_c$.)
Consider also the scale of Sobolev spaces $H^s_c:=H^s_{\C}\times H^s_{\C}$ where $H^s_{\C}\equiv H^s(\T,\C)$ is 
the Sobolev space of complex-valued distributions on $\T$ and $s\in\R$.
For any given $s\in\R$ the pairing \eqref{eq:pairing} induces an isomorphism $\imath_s : \big(H^s_c\big)'\to H^{-s}_c$ where
$(H^s_c)'$ denotes the space of continuous linear functionals on $H^s_c$. In this way, for any given $s\in\R$ the symplectic structure 
extends to a bounded bilinear map $\omega : H^s\times H^{-s}\to \C$.
The $L^2$-gradient $\partial_\varphi F=\big(\partial_{\varphi_1}F,\partial_{\varphi_2}F\big)$ of a $C^1$-function 
$F : H^s\to\C$ at $\varphi\in L^2_c$ is defined by $\partial_\varphi F:=\imath_s(d_\varphi F)\in H^{-s}_c$ where 
$d_\varphi F\in (H^s_c)'$ is the differential of $F$ at $\varphi\in L^2_c$. 
In particular, the Hamiltonian vector field $X_F$ corresponding to a $C^1$-smooth function 
$F : H^s_c\to\C$ at $\varphi\in H^s_c$ defined by the relation $\omega\big(\cdot,X_F(\varphi)\big)=d_\varphi F(\cdot)$
is then given by
\begin{equation}\label{eq:X_F}
X_F(\varphi)=i\big(-\partial_{\varphi_2}F,\partial_{\varphi_1}F\big).
\end{equation}
The vector field $X_F$ is a continuous map $X_F : H^s_c\to H^{-s}_c$.

\begin{Rem}
Since $X_F : H^s_c\to H^{-s}_c$ we see that strictly speaking $X_F$ is a weak vector field on $H^{-s}_c$.
However, for the sake of convenience in this paper we  will call such maps {\em vector fields} on $H^s_c$.
\end{Rem}

\begin{Rem}
The Poisson bracket of two $C^1$-smooth functions 
$F,G : H^s_c\to\C$ is then given by
\begin{equation}\label{eq:poisson_bracket}
\{F,G\}(\varphi):=d_\varphi F(X_G)=
-i\int_0^1\big((\partial_{\varphi_1}F)(\partial_{\varphi_2}G)-(\partial_{\varphi_1}G)(\partial_{\varphi_2}F)\big)\,dx
\end{equation}
provided that the pairing given by the integral in \eqref{eq:poisson_bracket} is well-defined.
\end{Rem}

The Hamiltonian $\HH : H^1_c\to\C$ of the NLS equation is
\begin{equation}\label{eq:NLS-Hamiltonian}
\HH(\varphi):=\int_0^1\big( \varphi_{1x}\varphi_{2x}+\varphi_1^2\varphi_2^2\big)\,dx.
\end{equation}
By \eqref{eq:X_F} the corresponding Hamiltonian vector field is
\begin{eqnarray}
X_{\HH}(\varphi)&=&i\big(-\partial_{\varphi_2}\HH,\partial_{\varphi_1}\HH\big)\nonumber\\
&=&i\big(\varphi_{1xx}-2\varphi_1^2\varphi_2,-\varphi_{2xx}+2\varphi_1\varphi_2^2\big).\label{eq:X_H}
\end{eqnarray}
Clearly, $X_\HH : H^2_c\to L^2_c$ is an analytic map. The NLS equation is then written as
\begin{equation}\label{eq:NLS-complex}
\left\{
\begin{array}{l}
\dot{\varphi}_1=i\varphi_{1xx}-2i\varphi_1^2\varphi_2,\\
\dot{\varphi}_2=-i\varphi_{2xx}+2i\varphi_1\varphi_2^2.
\end{array}
\right.
\end{equation}
The phase space $L^2_c$ has two real subspaces 
\begin{equation*}
L^2_r:=\big\{\varphi\in L^2_c\,\big|\,\varphi_2=\overline{\varphi_1}\big\}\quad\text{and}\quad
i L^2_r:=\big\{\varphi\in L^2_c\,\big|\,\varphi_2=-\overline{\varphi_1}\big\}
\end{equation*}
so that $L^2_c=L^2_2\oplus_\R i L^2_r$. For any $s\in\R$ one also defines in a similar way the real subspaces
$H^s_r$ and $i H^s_r$ in $H^s_c$ so that $H^s_c=H^s_r\oplus_\R i H^s_r$.
It follows from \eqref{eq:NLS-Hamiltonian} that the Hamiltonian $\HH$ is real valued when restricted to $H^1_r$ and $i H^1_r$.
Moreover, one easily sees from \eqref{eq:NLS-complex} that the Hamiltonian vector field $X_{\HH}$ is ``tangent'' to the real 
subspaces $H^1_r$ and $i H^1_r$ so that the restrictions
\[
X_\HH\big|_{H^2_r} : H^2_r\to L^2_r\quad\text {and}\quad X_\HH\big|_{i H^2_r} : i H^2_r\to i L^2_r\,.
\]
are well-defined, and hence real analytic maps. The vector field $X_\HH\big|_{H^2_r}$ corresponds to the defocusing NLS equation 
and the vector field $X_\HH\big|_{i H^2_r}$ corresponds to the focusing NLS equation. 
This is consistent with the fact that the restriction of the symplectic structure $\omega$ to $L^2_r$ and $i L^2_r$ is real valued.
For the sake of convenience in what follows we drop the restriction symbols in $X_\HH\big|_{i H^2_r}$ and
$X_\HH\big|_{H^2_r}$ and simply write $X_\HH$ instead.

\medskip\medskip

\noindent {\em 2) Constant potentials.} For any given complex number $c\in\C$, $c\ne 0$, consider the constant potential 
\[
\varphi_c(x):=(c,-\overline{c})\in i L^2_r\cap i C^\infty_r\,.
\]
It follows from \eqref{eq:X_H} that 
\begin{equation}\label{eq:X_H'}
X_\HH(\varphi_c)=2i|c|^2\big(c,\overline{c}\big).
\end{equation}
Since this vector does not vanish we see that $\varphi_c$ is {\em not} a critical point of the NLS Hamiltonian 
\eqref{eq:NLS-Hamiltonian} and hence $d_{\varphi_c}\HH\ne 0$ in $(H^1_c)'$.

\medskip\medskip

\noindent{\em 3) The re-normalized Hamiltonian.}
In addition to the NLS Hamiltonian \eqref{eq:NLS-Hamiltonian} consider the Hamiltonian
\begin{equation}\label{eq:H_1}
\HH_1(\varphi):=-\int_0^1\varphi_1(x)\varphi_2(x)\,dx\,.
\end{equation}
Note that this is the first Hamiltonian appearing in the NLS hierarchy -- see e.g. \cite{GK1}.
The corresponding Hamiltonian vector field is
\begin{equation}\label{eq:X_H1}
X_{\HH_1}(\varphi)=i \big(\varphi_1,-\varphi_2\big).
\end{equation}
For any $s\in\R$ we have that $X_{\HH_1} : H^s_c\to H^s_c$ and hence $X_{\HH_1}$ is a (regular) vector field on
$H^s_c$. This vector field is tangent to the real submanifolds $H^s_r$ and $i H^s_r$ and induces the following 
one-parameter group of diffeomorphisms of $i H^s_r$, 
\begin{equation}\label{eq:symmetry}
S^t : i H^s_r\to i H^s_r,\quad\quad
\big(\varphi_1^0,\varphi_2^0\big)\stackrel{S^t}{\mapsto}\big(\varphi_1^0 e^{i t},\varphi_2^0 e^{-i t}\big).
\footnote{In what follows we will restrict our attention to the real space $i H^s_r$, $s\in\R$.}
\end{equation}
The transformations \eqref{eq:symmetry} preserves the vector field $X_\HH$, i.e. 
for any $t\in\R$ and for any $\varphi\in i H^2_r$
\begin{equation}\label{eq:preservation_of_X_H}
S^t\big(X_\HH(\varphi)\big)=X_\HH(S^t(\varphi))\,.
\end{equation}
It follows from \eqref{eq:X_H'} and \eqref{eq:X_H1} that
\begin{equation}\label{eq:dependence_relation}
X_\HH(\varphi_c)=2|c|^2 X_{\HH_1}(\varphi_c).
\end{equation}
We have the following

\begin{Lem}\label{lem:normalized_hamiltonian}
Let $c\in\C\setminus\{0\}$. Then one has:
\begin{itemize}
\item[(i)] The re-normalized Hamiltonian $\HH^c : i H^1_r\to\R$,
\begin{equation}\label{eq:normalized_hamiltonian}
\HH^c(\varphi):=\HH(\varphi)-2|c|^2\HH_1(\varphi),\quad\varphi\in i H^1_r,
\end{equation}
has a critical point at $\varphi_c$.
\item[(ii)] The curve $\gamma_c :\R\to i H^1_r$,
\begin{equation}\label{eq:the_critical_circle}
\gamma_c : t\mapsto\big(c e^{2i|c|^2 t},-\overline{c} e^{-2i|c|^2 t}\big),\quad t\in\R,
\end{equation}
is a solution of the NLS equation \eqref{eq:NLS-complex} with initial data at $\varphi_c$. This is a time periodic solution with period
$\pi/|c|^2$.
\item[(iii)] The range of the curve $\gamma_c$ consists of critical points of the Hamiltonian $\HH^c$.
\end{itemize}
\end{Lem}

\begin{proofof}{Lemma \ref{lem:normalized_hamiltonian}}
Item $(i)$ follows directly from \eqref{eq:dependence_relation}. 
Since the symmetry \eqref{eq:symmetry} preserves both $X_\HH$ and $X_{\HH_1}$ we conclude from
\eqref{eq:dependence_relation} that for any $t\in\R$,
\begin{equation}\label{eq:dependence_relation'}
X_\HH\big(S^t(\varphi_c)\big)=2|c|^2 X_{\HH_1}\big(S^t(\varphi_c)\big).
\end{equation}
This together with the fact that $S^t(\varphi_c)$ is the integral curve of $X_{\HH_1}$ with initial data at $\varphi_c$
we conclude that $\gamma(t):=S_{2|c|^2t}(\varphi_c)$ is an integral curve of $X_\HH$ with initial data at $\varphi_c$.
This proves item $(ii)$. Item $(iii)$ follows from \eqref{eq:dependence_relation'}.
\end{proofof}

\section{The linearization of $X_{\HH^c}$ at $\varphi_c$ and its normal form}\label{sec:linearization}
In this Section we find the spectrum and the normal form of the linearized Hamiltonian vector field 
$X_{\HH^c} : i H^2_r\to i L^2_r$ at $\varphi_c\in i H^2_r$.
In view of Lemma \ref{lem:normalized_hamiltonian} the constant potential $\varphi_c$ is a singular point of
the vector field $X_{\HH^c}$, i.e. $X_{\HH^c}(\varphi_c)=0$. Moreover, by Lemma \ref{lem:normalized_hamiltonian} $(iii)$, 
the range of the periodic trajectory $\gamma_c$ consists of singular points of $X_{\HH^c}$.
It follows from \eqref{eq:X_H} and \eqref{eq:X_H1}  that for any $\varphi\in i H^2_r$,
\[
X_{\HH^c}(\varphi)= i
\left(
\begin{array}{c}
\varphi_{1xx}-2\varphi_1^2\varphi_2-2|c|^2\varphi_1\\
-\varphi_{2xx}+2\varphi_1\varphi_2^2+2|c|^2\varphi_2
\end{array}
\right).
\]
Hence, the linearized vector field $\big(dX_{\HH^c}\big)\big|_{\varphi=\varphi_c} : i H^2_r\to i L^2_r$ is given by
\begin{equation}\label{eq:linearization}
\big(dX_{\HH^c}\big)\big|_{\varphi=\varphi_c}
\left(
\begin{array}{c}
\delta\varphi_1\\
\delta\varphi_2
\end{array}
\right)
=i
\left(
\begin{array}{c}
(\delta\varphi_1)_{xx}+2|c|^2(\delta\varphi_1)-2c^2(\delta\varphi_2)\\
-(\delta\varphi_2)_{xx}+2\overline{c}^2(\delta\varphi_1)-2|c|^2(\delta\varphi_2)
\end{array}
\right)
\end{equation}
where $\big(\delta\varphi_1,\delta\varphi_2\big)\in i H^2_r$.
Since the symmetry \eqref{eq:symmetry} preserves $X_{\HH^c}$ and since for any $t\in\R$,
\[
S^t(\varphi_c)=\big(c e^{i t},-\overline{c}e^{-it}\big),
\]
the map $S^t$ conjugates the operator \eqref{eq:linearization} computed at $\varphi_c$ with the one 
computed at $\varphi_{c_t}$ with $c_t:=c e^{i t}$. More specifically, one has the following commutative diagram
\begin{equation}\label{eq:diagram}
\begin{tikzcd}
i H^2_r\arrow{r}{\LL_{c_t}}& i L^2_r\\
i H^2_r\arrow{u}{S^t}\arrow{r}{\LL_c}&i L^2_r\arrow[swap]{u}{S^t}
\end{tikzcd}
\end{equation}
where for simplicity of notation we denote
\[
\LL_c\equiv\big(dX_{\HH^c}\big)\big|_{\varphi=\varphi_c}.
\]
By choosing $t=-\arg(c)$ in the diagram above we obtain

\begin{Lem}\label{lem:c-real}
The operators $\LL_c$ and $\LL_{|c|}$ are conjugate.
\end{Lem}

With this in mind, in what follows we will assume without loss of generality that $c$ is real. In this case
\begin{equation}\label{eq:L_c}
\LL_c= i 
\begin{pmatrix}
1&0\\
0&-1
\end{pmatrix}
\partial_x^2+2 i c^2
\begin{pmatrix}
1&-1\\
1&-1
\end{pmatrix},
\quad c\in\R.
\end{equation}
For any $k\in\Z$ consider the vectors
\begin{equation}\label{eq:basis1}
\xi_k:=
\begin{pmatrix}
1\\
0
\end{pmatrix}
e^{2\pi i k x}
\quad\text{and}\quad
\eta_k:=
\begin{pmatrix}
0\\
1
\end{pmatrix}
e^{2\pi i k x}.
\end{equation}
The system of vectors $\big\{(\xi_k,\eta_k)\big\}_{k\in\Z}$ give an orthonormal basis in the complex Hilbert space $L^2_c$ so that
for any $\varphi=(\varphi_1,\varphi_2)\in L^2_c$,
\[
\varphi=\sum_{k\in\Z}\big(z_k\xi_k+w_k\eta_k\big),
\]
where $z_k:=\widehat{(\varphi_1)}_k$ and $w_k:=\widehat{(\varphi_2)}_k$. 
Denote $\ell^2_c:=\ell^2_\C\times\ell^2_\C$ where $\ell^2_\C\equiv\ell^2(\Z,\C)$ is the space of square summable sequences 
of complex numbers. In this way, $\big\{(z_k,w_k)\big\}_{k\in\Z}\in\ell^2_c$ are coordinates in $L^2_c$. In these coordinates,
the real subspace $i L^2_r$ is characterized by the condition that $\forall k\in\Z$, $w_k=-\overline{(z_{-k})}$, and the real subspace
$L^2_r$ is characterized by the condition that $\forall k\in\Z$, $w_k=\overline{(z_{-k})}$. We denote the corresponding spaces of 
sequences respectively by $i\ell^2_r$ and $\ell^2_r$.
It follows from \eqref{eq:symplectic_structure} that for any $k,l\in\Z$ one has 
\[
\omega\big(\xi_k,\xi_l\big)=\omega\big(\eta_k,\eta_l\big)=0\quad\text{and}\quad
\omega\big(\xi_k,\eta_{-l}\big)=-i\delta_{kl}
\]
where $\delta_{kl}$ is the Kronecker delta.
In addition to the vectors in \eqref{eq:basis1} consider for $k\in\Z$ the vectors
\begin{equation}\label{eq:basis2}
\xi_k':=\frac{1}{\sqrt{2}}\big(\xi_k-\eta_{-k}\big)=\frac{1}{\sqrt{2}}
\begin{pmatrix}
e^{2\pi k i x}\\
-e^{-2\pi k i x}
\end{pmatrix},
\quad
\eta_k':=\frac{i}{\sqrt{2}}\big(\xi_k+\eta_{-k}\big)= \frac{i}{\sqrt{2}}
\begin{pmatrix}
e^{2\pi k i x}\\
e^{-2\pi k i x}
\end{pmatrix}.
\end{equation}
Note that the system of vectors $\big\{\xi_k',\eta_k'\big\}_{k\in\Z}$ form an orthonormal basis in the real subspace $i L^2_r$. 
In addition, this is a {\em Darboux basis} in $i L^2_r$ with respect to the restriction of the symplectic structure 
\eqref{eq:symplectic_structure} to $i L^2_r$, i.e. for any $k,l\in\Z$ one has
\[
\omega\big(\xi_k',\xi_l'\big)=\omega\big(\eta_k',\eta_l'\big)=0\quad\text{and}\quad 
\omega\big(\xi_k',\eta_l'\big)=\delta_{kl}.
\] 
Moreover, for any $\varphi\in L^2_c$,
\[
\varphi=\sum_{k\in\Z}\big(x_k\xi_k'+y_k\eta_k' \big),
\]
where $\big\{(x_k,y_k)\}_{k\in\Z}$ are coordinates in $L^2_c$ so that the real subspace $i L^2_r$ is characterized by the condition
that $\big\{(x_k,y_k)\}_{k\in\Z}\in\ell^2(\Z,\R)\times\ell^2(\Z,\R)$. For any $k\in\Z$,
\[
x_k=\frac{1}{\sqrt{2}}\big(z_k-w_{-k}\big)\quad\text{and}\quad y_k=\frac{1}{i\sqrt{2}}\big(z_k+w_{-k}\big).
\]
Finally, consider the 2-(complex)dimensional subspaces in $L^2_c$,
\begin{equation}\label{eq:V_k}
V_k^\C:=\Span_\C\langle\xi_k,\eta_k \rangle,\quad k\in\Z,
\end{equation}
together with the 4-(real)dimensional {\em symplectic subspaces} in $i L^2_r$,
\begin{equation}\label{eq:W_k}
W_k^\R:=\Span_\R\langle\xi_k',\eta_k',\xi_{-k}',\eta_{-k}'\rangle,\quad k\in\Z_{\ge 1}.
\end{equation}
and
\begin{equation}\label{eq:W_0}
W_0^\R:=\Span_\R\langle\xi_0',\eta_0'\rangle.
\end{equation}
It follows from \eqref{eq:basis2} that for any $k\in\Z_{\ge 1}$,
\begin{equation}\label{eq:complexification_property}
W_k^\R\otimes\C=V_k^\C\oplus_\C V_{-k}^\C\quad\text{and}\quad W_0^\R\otimes\C=V_0^\C.
\end{equation}
Since $\{\xi_k',\eta_k' \}_{k\in\Z}$ is an orthonormal basis of $i L^2_r$,
\[
i L^2_r=\bigoplus_{k\in\Z_{\ge 0}}W_k^\R
\]
is a decomposition of $i L^2_r$ into $L^2$-orthogonal real subspaces.
We have the following

\begin{Th}\label{th:the_spectrum_of_X_H}
For any $c\in\R$, $c\notin \pi\Z$, the operator 
\[
\LL_c\equiv\big(dX_{\HH^c}\big)\big|_{\varphi=\varphi_c} : i H^2_r\to i L^2_r,
\]
has a compact resolvent. In particular, the spectrum of $\LL_c$ is discrete and has the following properties:
\begin{itemize}
\item[(i)] The spectrum of $\LL_c$ consists of $\lambda_0=0$ and
\begin{equation}\label{eq:the_spectrum_of_L_c}
\lambda_k=\left\{
\begin{array}{lc}
4\pi k\sqrt[+]{|c|^2-\pi^2k^2},&0<\pi|k|<|c|,\\
4\pi i k \sqrt[+]{\pi^2k^2-|c|^2},&\pi|k|>|c|,
\end{array}
\right.
\end{equation}
for any integer $k\in\Z\setminus\{0\}$.
The eigenvalue $\lambda_0$ has algebraic multiplicity two and geometric multiplicity one; for any $k\in\Z\setminus\{0\}$ 
the eigenvalue $\lambda_k$ has algebraic multiplicity two and geometric multiplicity two.

\item[(ii)] For any $k\in\Z$ the complex linear space $V_k^\C$ (see \eqref{eq:V_k}) is an invariant space of $\LL_c$ 
in $L^2_c$ and 
\begin{equation}\label{eq:L_k}
\LL_c\big|_{V_k^\C}= i
\begin{pmatrix}
2|c|^2-4\pi^2k^2&-2|c|^2\\
2|c|^2&4\pi^2 k^2-2|c|^2
\end{pmatrix}
\end{equation}
in the basis of $V_k^\C$ given by $\xi_k$ and $\eta_k$. If $0<\pi |k|<|c|$ the matrix \eqref{eq:L_k} has two real eigenvalues 
$\pm4\pi|k|\sqrt[+]{|c|^2-\pi^2k^2}$ and  if $|c|<|k|\pi$ it has two purely imaginary complex eigenvalues
$\pm4\pi i |k|\sqrt[+]{\pi^2k^2-|c|^2}$.

\item[(iii)] For any $k\in\Z_{\ge 1}$ the real symplectic space $W_k^\R$ (see \eqref{eq:W_k}) is an invariant space
of $\LL_c$ in $i L^2_r$. When written in the basis $\big\{\xi_k,\eta_k,\xi_{-k},\eta_{-k}\big\}$ of
the complexification of $W_k^\R$ the matrix representation of of the operator $\LL_c\big|_{W_k^\R}$ consists of 
two diagonal square blocks of the form \eqref{eq:L_k}. The real symplectic space $W_0^\R$ is an invariant space of 
the operator $\LL_c$ in $i L^2_r$ and
\[
\LL_c\big|_{W_0^\R}=2i |c|^2
\begin{pmatrix}
1&-1\\
1&-1
\end{pmatrix}
\]
when written in the basis $\big\{\xi_0,\eta_0\big\}$ of the complexification of $W_0^\R$. Zero is a double eigenvalue of this
matrix with geometric multiplicity one.
\end{itemize}
\end{Th}
\begin{proofof}{Theorem \ref{th:the_spectrum_of_X_H}}
The proof of this Theorem follows directly from the matrix representation of $\LL_c$ when computed in the basis
of $V_k^\C$ given by the vectors $\xi_k$ and $\eta_k$.
\end{proofof}

Theorem \ref{th:the_spectrum_of_X_H} implies that the linearized vector field $(dX_{\HH^c})\big|_{\varphi=\varphi_c}$
has the following {\em normal form}. 

\begin{Th}\label{th:normal_form}
Assume that $c\in\R$ and $c\notin \pi\Z$. We have:
\begin{itemize} 
\item[(i)] The vectors $\alpha_0:=\xi_0'$ and $\beta_0:=\eta_0'$ form a Darboux basis of $W_0^\R\subseteq i L^2_r$ such that
\begin{equation}\label{eq:L-drift}
\LL_c\big|_{W_0^\R}=4|c|^2
\begin{pmatrix}
0&0\\
1&0
\end{pmatrix}.
\end{equation}
\item[(ii)] For any $k\in\Z_{\ge 1}$ there exists a Darboux basis $\big\{\alpha_k,\beta_k,\alpha_{-k},\beta_{-k}\big\}$ 
in $W_k^\R\subseteq i L^2_r$ such that\footnote{Here $\omega(\alpha_k,\beta_k)=\omega(\alpha_{-k},\beta_{-k})=1$ 
while all other skew-symmetric products between these vectors vanish.} for $0<\pi k<|c|$,
\begin{equation}\label{eq:L-focus_focus}
\LL_c\big|_{W_k^\R}=4\pi k\sqrt[+]{|c|^2-\pi^2k^2}
\begin{pmatrix}
1&0&0&0\\
0&-1&0&0\\
0&0&1&0\\
0&0&0&-1
\end{pmatrix},
\end{equation}
and for $\pi k>|c|$,
\begin{equation}\label{eq:L-center}
\LL_c\big|_{W_k^\R}=4\pi k\sqrt[+]{\pi^2k^2-|c|^2}
\begin{pmatrix}
0&1&0&0\\
-1&0&0&0\\
0&0&0&1\\
0&0&-1&0
\end{pmatrix}.
\end{equation}
In addition, one has the following uniform in $k\in\Z$ with $\pi k>|c|$ estimates
\begin{equation}\label{eq:basis_asymptotics}
\left\{
\begin{array}{l}
\alpha_k=\xi_k'+O(1/k^2),\quad\beta_k=\eta_k'+O(1/k^2),\\
\alpha_{-k}=\xi_{-k}'+O(1/k^2),\quad\beta_{-k}=\eta_{-k}'+O(1/k^2),
\end{array}
\right.
\end{equation}
where  $\big\{\xi_k',\eta_k'\}_{k\in\Z}$ is the orthonormal Darboux basis \eqref{eq:basis2} in $i L^2_r$.
\end{itemize}
\end{Th}
Recall that $\h^s(\Z,\R)$ with $s\in\R$ denotes the Hilbert space of sequences of real numbers $(a_k)_{k\in\Z}$ so that 
$\sum_{k\in\Z}\langle k\rangle^{2s}|a_k|^2<\infty$ where $\langle k\rangle:=\sqrt{1+|k|^2}$.
We will also need the Banach space $\ell^1_s(\Z,\R)$ of sequences of real numbers $(a_k)_{k\in\Z}$ so that
$\sum_{k\in\Z}\langle k\rangle^s|a_k|<\infty$.

\begin{Rem}\label{rem:darboux_basis}
Note that the asymptotics \eqref{eq:basis_asymptotics} imply (see e.g. \cite[Section 22.5]{Lax}) that for any $s\in\R$
the system $\big\{\alpha_k,\beta_k\big\}_{k\in\Z}$ is a basis in $i H^s_r$ in the sense that
for any $\varphi\in i H^s_r$ there exists a unique sequence $\big\{(p_k,q_k)\big\}_{k\in\Z}$ in
$\h^s(\Z,\R)\times\h^s(\Z,\R)$ such that $\varphi=\sum_{k\in\Z}\big(p_k\alpha_k+q_k\beta_k\big)$ where 
the series converges in $i H^s_r$ and the mapping
\[
i H^s_r\to\h^s(\Z,\R)\times\h^s(\Z,\R),\quad\varphi\mapsto\big\{(p_k,q_k)\big\}_{k\in\Z},
\]
is an isomorphism. Note that $\big\{\alpha_k,\beta_k\big\}_{k\in\Z}$ is a Darboux basis that is {\em not} an orthonormal
basis in $i L^2_r$.
\end{Rem}

\begin{proofof}{Theorem \ref{th:normal_form}}
Item $(i)$ follows by a direct computation in the basis of $W_0^\R$ provided by the vectors 
$\alpha_0:=\xi_0'$ and $\beta_0:=\eta_0'$ (see \eqref{eq:basis2}).
Towards proving item $(ii)$, we first consider the case when $\pi k>|c|$. Denote for simplicity
\[
L_{\pm k}:=\LL_c\big|_{V_{\pm k}^\C},\quad a_k:=4\pi^2k^2-2|c|^2,\quad b:=2|c|^2,
\]
and note that $a_k^2-b^2=16\pi^2 k^2\big(\pi^2 k^2-|c|^2\big)>0$.
Then, in view of \eqref{eq:L_k},
\[
L_k= i
\begin{pmatrix}
-a_k&-b\\
b&a_k
\end{pmatrix}
\quad\text{and}\quad
L_{-k}= L_k
\]
in the basis of $V_m^\C$ given by $\big\{\xi_m,\eta_m\big\}$ for $m=\pm k$.
Denote, by $\varkappa_k$ the positive square root of the quantity
\[
\varkappa_k^2=\sqrt[+]{a_k^2-b^2}\Big(a_k+\sqrt[+]{a_k^2-b^2}\Big).
\]
It follows from Theorem \ref{th:the_spectrum_of_X_H} $(ii)$ and \eqref{eq:L_k} that 
\begin{equation}\label{eq:eigenvectorF+}
F_k:=-\frac{1}{\varkappa_k}
\begin{pmatrix}
-b\\
a_k+\sqrt[+]{a_k^2-b^2}
\end{pmatrix}
e^{-2\pi i k x}
\end{equation}
and
\begin{equation}\label{eq:eigenvectorF-}
F_{-k}:=-\frac{1}{\varkappa_k}
\begin{pmatrix}
-b\\
a_k+\sqrt[+]{a_k^2-b^2}
\end{pmatrix}
e^{2\pi i k x}
\end{equation}
are linearly independent eigenfunctions of the restriction of $\LL_c$ to the invariant space 
$V_k^\C\oplus_\C V_{-k}^\C=W_k^\R\otimes\C$ with eigenvalue
\[
\lambda_k=i\sqrt[+]{a_k^2-b^2}=4\pi i k\sqrt[+]{\pi^2 k^2-|c|^2}.
\]
The eigenfunctions have been normalized in a way convenient for our purposes.
Denote by $\sigma$ the {\em complex conjugation} in $L^2_c$ corresponding to the real subspace $i L^2_r$,
\begin{equation}\label{eq:conjugation1}
\sigma : L^2_c\to L^2_c,\quad\quad(\varphi_1,\varphi_2)\mapsto(-\overline{\varphi_2},-\overline{\varphi_1}).
\end{equation}

\begin{Rem}\label{rem:complex_conjugation}
One easily sees from \eqref{eq:conjugation1} that $\sigma\big|_{i L^2_r}=\id_{i L^2_r}$ and 
$\sigma\big|_{L^2_r}=-\id_{L^2_r}$. This implies that for any $\alpha,\beta\in i L^2_r$,
\begin{equation}\label{eq:conjugation2}
\sigma(\alpha+i\beta)=\alpha-i\beta.
\end{equation} 
Since the operator $\LL_c$ is real (i.e., $\LL_c :  i H^2_r\to i L^2_r$) and complex-linear, we conclude from \eqref{eq:conjugation2} 
that if $f\in H^2_c$ is an eigenfunction of $\LL_c$ with eigenvalue $\lambda\in\C$ then $\sigma(f)$ is an eigenfunction of
$\LL_c$ with eigenvalue $\overline{\lambda}$. Moreover, one easily checks that for any $f,g\in H^2_c$,
\begin{equation}\label{eq:simplectic_transform}
\omega\big(\LL_c f,g\big)=-\omega\big(f,\LL_c g\big),
\end{equation}
which is consistent with the fact that $X_{\HH^c}$ is a Hamiltonian vector field. Equation \eqref{eq:simplectic_transform}
implies that if $f,g\in H^2_c$ are eigenfunctions of $\LL_c$ with the same eigenvalue $\lambda\in\C\setminus\{0\}$ then 
they are isotropic, i.e. $\omega(f,g)=0$.
\end{Rem}

Since $L^2_c=i L^2_r\oplus_\R L^2_r$ we have that
\begin{equation}\label{eq:the_basis_center}
F_k=\alpha_k+i\beta_k\quad\text{and}\quad F_{-k}=\alpha_{-k}+i\beta_{-k}
\end{equation}
where by \eqref{eq:conjugation2}
\[
\alpha_{\pm k}:=\frac{F_{\pm k}+\sigma(F_{\pm k})}{2}
\quad\text{and}\quad
\beta_{\pm k}:=\frac{F_{\pm k}-\sigma(F_{\pm k})}{2i}
\]
are elements in $i L^2_r$. Moreover, in view of Remark \ref{rem:complex_conjugation},
\begin{equation}\label{eq:eigenvectorG+}
G_k:=\sigma(F_k)=\frac{1}{\varkappa_k}
\begin{pmatrix}
a_k+\sqrt[+]{a_k^2-b^2}\\
-b
\end{pmatrix}
e^{2\pi i k x}
\end{equation}
and
\begin{equation}\label{eq:eigenvectorG-}
G_{-k}:=\sigma(F_{-k})=\frac{1}{\varkappa_k}
\begin{pmatrix}
a_k+\sqrt[+]{a_k^2-b^2}\\
-b
\end{pmatrix}
e^{-2\pi i k x}
\end{equation}
are linearly independent eigenfunctions of the restriction of $\LL_c$ to the invariant space 
$V_k^\C\oplus_\C V_{-k}^\C=W_k^\R\otimes\C$ with eigenvalue
\[
-\lambda_k=-i\sqrt[+]{a_k^2-b^2}=-4\pi i k\sqrt[+]{\pi^2 k^2-|c|^2}.
\]
It follows from \eqref{eq:symplectic_structure}, \eqref{eq:eigenvectorF+}, \eqref{eq:eigenvectorF-}, \eqref{eq:eigenvectorG+}, 
and \eqref{eq:eigenvectorG-} that
\begin{equation}
\omega\big(F_k,G_k\big)=\omega\big(F_{-k},G_{-k}\big)=-2 i\quad\text{and}\quad
\omega\big(F_k,G_{-k}\big)=\omega\big(F_{-k},G_k\big)=0
\end{equation}
while $\omega\big(F_k,F_{-k}\big)=\omega\big(G_k,G_{-k}\big)=0$ in view of Remark \ref{rem:complex_conjugation}. 
Hence,
\[
\omega\big(\alpha_k,\beta_k\big)=\frac{1}{4i}\,\omega\big(F_k+\sigma(F_k),F_k-\sigma(F_k)\big)=1
\]
and similarly $\omega(\alpha_{-k},\beta_{-k})=1$ whereas all other values of $\omega$
evaluated at pairs of vectors from the set $\big\{\alpha_k,\beta_k,\alpha_{-k},\beta_{-k}\big\}$ vanish.
This shows that the vectors $\big\{\alpha_k,\beta_k,\alpha_{-k},\beta_{-k}\big\}$ form a Darboux basis. 
The matrix representation \eqref{eq:L-center} of $\LL_c$ in this basis then follows from \eqref{eq:the_basis_center}
and the fact that $F_k$ and $F_{-k}$  are eigenfunctions of $\LL_c$ with eigenvalue $i \sqrt[+]{a_k^2-b^2}$,
\[
\LL_c\big(\alpha_k+i\beta_k\big)=i \sqrt[+]{a_k^2-b^2}\,\big(\alpha_k+i\beta_k\big)\quad\text{and}\quad
\LL_c\big(\alpha_{-k}+i\beta_{-k}\big)=i \sqrt[+]{a_k^2-b^2}\,\big(\alpha_{-k}+i\beta_{-k}\big).
\]
The asymptotic relations in \eqref{eq:basis_asymptotics} follow from the explicit formulas for $F_{\pm k}$ and $G_{\pm k}$ above
together with $\alpha_m=\big(F_m+G_m\big)/2$, $\beta_m=\big(F_m-G_m\big)/2i$ with $m=\pm k$, and 
$a_k=4\pi^2k^2-2|c|^2$. 

\medskip

The case when $0<\pi |k|<|c|$ is treated in a similar way. In fact, take $k\in\Z$ with $0<\pi k<|c|$ and
denote by $\varkappa_k$ the branch of the square root of
\[
\varkappa_k^2=\sqrt[+]{b^2-a_k^2}\Big(a_k-i \sqrt{b^2-a_k^2}\Big)
\]
that lies in the fourth quadrant of the complex plane $\C$.
It follows from Theorem \ref{th:the_spectrum_of_X_H} $(ii)$ and \eqref{eq:L_k} that 
\begin{equation}\label{eq:eigenvectorF'+}
F_k:=\frac{1}{\varkappa_k}
\begin{pmatrix}
-b\\
a_k-i \sqrt[+]{b^2-a_k^2}
\end{pmatrix}
e^{2\pi i k x}
\end{equation}
and
\begin{equation}\label{eq:eigenvectorF'-}
F_{-k}:=\sigma\big(F_k\big)=-\frac{1}{\overline{\varkappa_k}}
\begin{pmatrix}
a_k+i \sqrt[+]{b^2-a_k^2}\\
-b
\end{pmatrix}
e^{-2\pi i k x}
\end{equation}
are linearly independent eigenfunctions of the restriction of $\LL_c$ to the invariant space 
$V_k^\C\oplus_\C V_{-k}^\C=W_k^\R\otimes\C$ with eigenvalue
\[
\lambda_k=\sqrt[+]{b^2-a_k^2}=4\pi k\sqrt[+]{|c|^2-\pi^2 k^2}.
\]
By arguing in the same way as above, one sees that
\begin{equation}\label{eq:eigenvectorG'+}
G_k:=-\frac{1}{\overline{\varkappa_k}}
\begin{pmatrix}
-b\\
a_k+i \sqrt[+]{b^2-a_k^2}
\end{pmatrix}
e^{2\pi i k x}
\end{equation}
and
\begin{equation}\label{eq:eigenvectorG'-}
G_{-k}:=\sigma\big(G_k\big)=\frac{1}{\varkappa_k}
\begin{pmatrix}
a_k- i \sqrt[+]{b^2-a_k^2}\\
-b
\end{pmatrix}
e^{-2\pi i k x}
\end{equation}
are linearly independent eigenfunctions of the restriction of $\LL_c$ to the invariant space 
$V_k^\C\oplus_\C V_{-k}^\C=W_k^\R\otimes\C$ with eigenvalue
\[
-\lambda_k=-\sqrt[+]{b^2-a_k^2}=-4\pi k\sqrt[+]{|c|^2-\pi^2 k^2}.
\]
Since $L^2_c=i L^2_r\oplus_\R L^2_r$ we have that
\begin{equation}\label{eq:the_basis_focus_focus1}
F_k=\alpha_k+i\alpha_{-k}\quad\text{and}\quad G_k=\beta_k+i\beta_{-k},
\end{equation}
where
\[
\alpha_{\pm k}:=\frac{F_{\pm k}+\sigma(F_{\pm k})}{2}
\quad\text{and}\quad
\beta_{\pm k}:=\frac{F_{\pm k}-\sigma(F_{\pm k})}{2i}
\]
are elements in $i L^2_r$. By \eqref{eq:conjugation2} this implies that 
\begin{equation}\label{eq:the_basis_focus_focus2}
F_{-k}=\alpha_k-i\alpha_{-k}\quad\text{and}\quad G_{-k}=\beta_k-i\beta_{-k}.
\end{equation}
It follows from \eqref{eq:symplectic_structure}, \eqref{eq:eigenvectorF'+}, \eqref{eq:eigenvectorF'-}, \eqref{eq:eigenvectorG'+}, 
and \eqref{eq:eigenvectorG'-} that
\begin{equation}
\omega\big(F_k,G_k\big)=\omega\big(F_{-k},G_{-k}\big)=0\quad\text{and}\quad
\omega\big(F_k,G_{-k}\big)=\omega\big(F_{-k},G_k\big)=2
\end{equation}
while $\omega\big(F_k,F_{-k}\big)=\omega\big(G_k,G_{-k}\big)=0$ in view of Remark \ref{rem:complex_conjugation}.
This together with \eqref{eq:the_basis_focus_focus1} and \eqref{eq:the_basis_focus_focus2} implies that the vectors 
$\big\{\alpha_k,\beta_k,\alpha_{-k},\beta_{-k}\big\}$ form a Darboux basis in $W_k^\R$. 
The matrix representation \eqref{eq:L-focus_focus} of $\LL_c$ in this basis then follows from 
\eqref{eq:the_basis_focus_focus1}, \eqref{eq:the_basis_focus_focus2}, and the fact that $F_k$ and $G_k$ given by
\eqref{eq:eigenvectorF'+} and \eqref{eq:eigenvectorG'+} are eigenfunctions of $\LL_c$ with (real) eigenvalues 
$\pm\sqrt[+]{b^2-a_k^2}$,
\[
\LL_c\big(\alpha_k+i\alpha_{-k}\big)=\sqrt[+]{b^2-a_k^2}\,\big(\alpha_k+i\alpha_{-k}\big)\quad\text{and}\quad
\LL_c\big(\beta_k+i\beta_{-k}\big)=- \sqrt[+]{b^2-a_k^2}\,\big(\beta_k+i\beta_{-k}\big).
\]
This completes the proof of Theorem \ref{th:normal_form}.
\end{proofof}

\begin{Rem}
In fact, the canonical form \eqref{eq:L-drift}, \eqref{eq:L-focus_focus}, and \eqref{eq:L-center}, of the restriction of the
operator $\LL_c$ to the invariant symplectic space $W_k^\R$ with $k\ge 0$ can be deduced from the description of the spectrum
of $\LL_c$ obtained in Theorem \ref{th:the_spectrum_of_X_H} $(i)$, $(ii)$, and the Williamson classification of linear Hamiltonian
systems in $\R^{2n}$ (see \cite{Arn,Wi}). Instead of doing this, we choose to construct the normalizing Darboux basis directly. 
The reason is twofold: first, in this way we obtain explicit formulas for the normalizing basis, and second, we need the asymptotic 
relations \eqref{eq:basis_asymptotics} to conclude that the system of vectors $\big\{\alpha_0,\beta_0\big\}$ together with
$\big\{\alpha_k,\beta_k,\alpha_{-k},\beta_{-k}\big\}_{k\in\Z_{\ge 1}}$ form a Darboux basis in $i L^2_r$ in the sense described 
in Remark \ref{rem:darboux_basis}.
\end{Rem}

In this way, as a consequence of Theorem \ref{th:normal_form} we obtain the following instance of an infinite dimensional 
version of the Williamson classification of linear Hamiltonian systems in $\R^{2n}$(\cite{Wi}).

\begin{Th}\label{th:hessian_normal_form}
Assume that $c\in\R$ and $c\notin \pi\Z$. Then the Hessian $d^2_{\varphi_c}\HH^c$, when viewed as a quadratic form represented
in the Darboux basis $\big\{\alpha_k,\beta_k\big\}_{k\in\Z}$ in $i L^2_r$ given by Theorem \ref{th:normal_form}, takes the form
\begin{eqnarray}\label{eq:hessian_normal_form}
d^2_{\varphi_c}\HH^c&=&4|c|^2dp_0^2-
\sum_{0<\pi k<|c|}4\pi k\sqrt[+]{|c|^2-\pi^2k^2}\big(dp_kdq_k+dp_{-k}dq_{-k}\big)\nonumber\\
&-&\sum_{\pi|k|>|c|}4\pi|k|\sqrt[+]{\pi^2k^2-|c|^2}\big(dp_k^2+dq_k^2\big)
\end{eqnarray}
where $\big\{(dp_k,dq_k)\big\}_{k\in\Z}$ are the dual coordinates in this basis.
\end{Th}

For any $0<\pi k<|c|$ denote
\begin{equation}\label{eq:focus-focus}
I_k:=p_kq_k+p_{-k}q_{-k}\quad\text{and}\quad I_{-k}:=p_kq_{-k}-p_{-k}q_k,
\end{equation}
and for $\pi|k|>|c|$,
\[
I_k:=\big(p_k^2+q_k^2\big)/2
\]
whereas for $k=0$
\[
I_0:=p_0^2/2\,.
\]
Note that the functions in \eqref{eq:focus-focus} are the commuting integrals characterizing the {\em focus-focus} singularity 
in the symplectic space $\R^4$  --  see e.g. \cite{Zung}. We conjecture that the following holds: 
{\em There exists an open neighborhood $U$ of $\varphi_c$ in $i L^2_r$, an open neighborhood $V$ of zero in 
$\ell^2(\Z,\R)\times\ell^2(\Z,\R)$, and a {\em canonical} real analytic diffeomorphism $\Phi : U\to V$ such that for any $s\ge 0$,
\[
\Phi : U\cap i H^s_r\to V\cap\big(\h^s(\Z,\R)\times\h^s(\Z,\R)\big),\quad\varphi\mapsto\big\{(p_k,q_k)\big\}_{k\in\Z},
\]
and for any $(p,q)\in V\cap\big(\h^1(\Z,\R)\times\h^1(\Z,\R)\big)$,
\[
\HH^c\circ\Phi^{-1}(p,q)=H^c\big(\{I_k\}_{k\in\Z}\big),
\]
where $H^c : \ell^1_2(\Z,\R)\to\R$ is a real analytic map.
We will discuss this conjecture in future work.}

\section{Non-existence of local Birkhoff coordinates}\label{sec:birkhoff_coordinates}
First, we will discuss the notion of {\em local Birkhoff coordinates}.
Let $\h^s\equiv\h^s(\Z,\R)$ and $\ell^2\equiv\ell^2(\Z,\R)$.

\begin{Def}\label{def:local_birkhoff_coordinates}
We say that the focusing NLS equation has {\em local Birkhoff coordinates} in a neighborhood of $\varphi^\bullet\in i H^2_r$
if there exist an open connected neighborhood $U$ of $\varphi^\bullet$ in $i L^2_r$, an open neighborhood $V$ of 
$(p^\bullet,q^\bullet)\in\h^2\times\h^2$ in $\ell^2\times\ell^2$, 
and a canonical $C^2$-diffeomorphism $\Phi : U\to V$ such that for any $0\le s\le 2$,
\[
\Phi : U\cap i H^s_r\to V\cap\big(\h^s\times\h^s\big),\quad\varphi\mapsto\big\{(p_k,q_k)\big\}_{k\in\Z},
\]
is a $C^2$-diffeomorphism and for any $k\in\Z$ the Poisson bracket $\{I_k,H\}$, where $H:=\HH\circ\Phi^{-1}$ and 
$I_k:=\big(p_k^2+q_k^2\big)/2$, vanishes on $V\cap\big(\h^1\times\h^1\big)$.
\end{Def}
The map $\Phi : U\to V$ being {\em canonical} means that 
\begin{equation}\label{eq:canonical}
(\Phi^{-1})^*\omega=\sum_{k\in\Z}dp_k\wedge dq_k
\end{equation}
where $\Phi^{-1} : V\to U$ is the inverse of $\Phi : U\to V$ and $\omega$ is symplectic form \eqref{eq:symplectic_structure} 
on $i L^2_r$. Assume that the focusing NLS equation has local Birkhoff coordinates in a neighborhood of $\varphi^\bullet\in i H^2_r$.
Then, for any $k\in\Z$ and $0\le s\le 2$ consider the action variable
\[
{\mathcal I}_k :  U\cap i H^s_r\to\R,\quad{\mathcal I}_k:=I_k\circ\Phi.
\]
Recall that $\big\{S^t\big\}_{t\in\R}$ denotes the Hamiltonian flow,
\[
S^t : i H^s_r\to i H^s_r,\quad(\varphi_1,\varphi_2)\mapsto\big(\varphi_1 e^{it},\varphi_2 e^{-it}\big),
\]
generated by the Hamiltonian $\HH_1(\varphi)=-\int_0^1\varphi_1(x)\varphi_2(x)\,dx$ (see \eqref{eq:H_1})
from the standard NLS hierarchy (see e.g. \cite{GK1}).

\begin{Def}\label{def:gauge_invariant_coordinates}
The local Birkhoff coordinates are called {\em gauge invariant} if for any $k\in\Z$, $0\le s\le 2$, and for any 
$\varphi\in U\cap i H^s_r$ and $t\in\R$ such that $S^t(\varphi)\in U\cap i H^s_r$ one has 
${\mathcal I}_k\big(S^t(\varphi)\big)={\mathcal I}_k(\varphi)$. 
\end{Def}

\begin{Rem}
The gauge invariance of local Birkhoff coordinates means that the Hamiltonian $\mathcal{H}_1$ belongs to the Poisson algebra 
$\mathcal{A}_{\mathcal{I}}:=\big\{F\in C^1(U,\C)\,\big|\,\{F,\mathcal{I}_k\}=0\,\,\forall k\in\Z\big\}$
generated by the local action variables $\big\{\mathcal{I}_k\big\}_{k\in\Z}$ . Note that, for example, $\HH_1$ belongs
to the Poisson algebra generated by the functionals $\big\{\Delta_\lambda\big\}_{\lambda\in\C}$ where 
$\Delta_\lambda : i L^2_r\to\C$ is the discriminant 
$\Delta_\lambda(\varphi)\equiv\Delta(\lambda,\varphi):=\tr M(x,\lambda,\varphi)|_{x=1}$ and $M(x,\lambda,\varphi)$ is 
the fundamental $2\times 2$-matrix solution of the Zakharov-Shabat system (see e.g. \cite{GK1}).
\end{Rem}

The main result of this Section is Theorem \ref{th:main_introduction} stated in the Introduction which we recall for the
convenience of the reader.

\begin{Th}\label{th:main}
For any given $c\in\C$ with $|c|\notin\pi\Z$ and $|c|>\pi$ the focusing NLS equation does {\em not} admit gauge invariant 
local Birkhoff coordinates in a neighborhood of the constant potential $\varphi_c\in i C^\infty_r$.
\end{Th}

Consider the commutative diagram
\begin{equation}\label{eq:diagram_vector_field}
\begin{tikzcd}
U\cap i H^2_r\arrow{r}{X_{\HH^c}}\arrow[swap]{d}{\Phi}& i L^2_r\arrow{d}{\Phi_*}\\
V\cap\big(\h^2\times\h^2\big)\arrow{r}{\widetilde{X}_{\HH^c}}&\ell^2\times\ell^2
\end{tikzcd}
\end{equation}
where $X_{\HH^c}$ is the Hamiltonian vector field of the re-normalized Hamiltonian 
$\HH^c$, $\Phi_*\equiv (d\Phi)|_{\varphi=\varphi_c}$, and $\widetilde{X}_{\HH^c}$ is defined by the diagram. 
By linearizing the maps in this diagram at $\varphi_c$ we obtain
\begin{equation}\label{eq:diagram_operators}
\begin{tikzcd}
i H^2_r\arrow{r}{\LL_c}\arrow[swap]{d}{\Phi_*}& i L^2_r\arrow{d}{\Phi_*}\\
\h^2\times\h^2\arrow{r}{\widetilde{\LL}_c}&\ell^2\times\ell^2
\end{tikzcd}
\end{equation}
where $\LL_c$ is the linearization of $X_{\HH^c}$ at the critical point $\varphi_c$ and $\widetilde{\LL}_c$ is the linearization
of $\widetilde{X}_{\HH^c}$ at the critical point $(p^\bullet,q^\bullet)=\Phi(\varphi_c)$. In particular, we see 
that the (unbounded) linear operator $\LL_c$ on $i L^2_r$ with domain $i H^2_r$ is conjugated to the operator
$\widetilde{\LL}_c$ on $\ell^2\times\ell^2$ with domain $\h^2\times\h^2$. We have

\begin{Lem}\label{lem:the_spectrum_of_X_H_tilde}
Assume that for a given $c\in\C$ the focusing NLS equation has gauge invariant local Birkhoff coordinates in a neighborhood
of the constant potential $\varphi_c$. Then the spectrum of the operator $\widetilde{\LL}_c$ is discrete and lies on the 
imaginary axis.
\end{Lem}

\begin{proof}[Proof of Lemma \ref{lem:the_spectrum_of_X_H_tilde}]
Let $\big\{(p_k,q_k)\big\}_{k\in\Z}$ be the local Birkhoff coordinates on $V\cap\big(\ell^2\times\ell^2\big)$ and
let $H^c :=\HH^c\circ\Phi^{-1} : V\cap\big(\h^1\times\h^1\big)\to\R$ be the Hamiltonian $\HH^c$ in these coordinates.
One easily concludes from \eqref{eq:canonical} and \eqref{eq:diagram_vector_field} that in the open neighborhood 
$V\cap\big(\h^2\times\h^2\big)$ of the critical point $z^\bullet:=(p^\bullet,q^\bullet)$ one has
\[
\widetilde{X}_{\HH^c}=X_{H^c}=\sum_{n\in\Z}\Big(\frac{\partial H^c}{\partial p_n}\,\partial_{q_n}-
\frac{\partial H^c}{\partial q_n}\,\partial_{p_n}\Big).
\]
Since by Lemma \ref{lem:normalized_hamiltonian}  and \eqref{eq:diagram_vector_field}, 
$z^\bullet$ is a critical point of $\widetilde{X}_{\HH^c}$, 
\begin{equation}\label{eq:critical_point}
\frac{\partial H^c}{\partial p_n}\Big|_{z\bullet}=\frac{\partial H^c}{\partial q_n}\Big|_{z^\bullet}=0
\end{equation}
for any $n\in\Z$. In addition, we obtain that the operator $\widetilde{\LL}_c : \h^2\times\h^2\to\ell^2\times\ell^2$
takes the form
\begin{eqnarray}\label{eq:L-tilde}
\widetilde{\LL}_c&\equiv& d_{z^\bullet}\widetilde{X}_{\HH^c}\nonumber\\
&=&\sum\limits_{n\in\Z}
\partial_{q_n}\otimes\sum\limits_{l\in\Z}\Big(\frac{\partial^2 H^c}{\partial p_n\partial p_l}\,dp_l+
\frac{\partial^2 H^c}{\partial p_n\partial q_l}\,dq_l\Big)\Big|_{z^\bullet}\\
&-&\sum\limits_{n\in\Z}
\partial_{p_n}\otimes\sum\limits_{l\in\Z}\Big(\frac{\partial^2 H^c}{\partial q_n\partial p_l}\,dp_l+
\frac{\partial^2 H^c}{\partial q_n\partial q_l}\,dq_l\Big)\Big|_{z^\bullet}.\nonumber
\end{eqnarray}
Note that for any $l\in\Z$,
\[
X_{I_l}=p_l\partial_{q_l}-q_l\partial_{p_l}.
\]
Since the local Birkhoff coordinates are assumed gauge invariant and since $dH(X_{I_k})=\{H,I_k\}=0$ 
for any $k\in\Z$, we obtain that for any $l\in\Z$,
\begin{equation}\label{eq:integrals}
0=(d H^c)(X_{I_l})=p_l\frac{\partial H^c}{\partial q_l}-q_l\frac{\partial H^c}{\partial p_l}
\end{equation}
in the open neighborhood $V\cap\big(\h^2\times\h^2\big)$ of $z^\bullet$.
By taking the partial derivatives $\partial_{p_n}$ and $\partial_{q_n}$ of the equality above 
at $z^\bullet$ for $n\in\Z$ we obtain, in view of \eqref{eq:critical_point}, that for any $n,l\in\Z$,
\begin{equation}\label{eq:relation1}
\Big(\frac{\partial^2 H^c}{\partial p_n\partial q_l}\,p_l-
\frac{\partial^2 H^c}{\partial p_n\partial p_l}\,q_l\Big)\Big|_{z^\bullet}=0
\quad\text{and}\quad
\Big(\frac{\partial^2 H^c}{\partial q_n\partial q_l}\,p_l-
\frac{\partial^2 H^c}{\partial q_n\partial p_l}\,q_l\Big)\Big|_{z^\bullet}=0.
\end{equation}
We split the set of indices $\Z$ in the sum above into two subsets
\[
A:=\big\{l\in\Z\,\big|\,(p^\bullet_l,q^\bullet_l)\ne(0,0)\big\}\quad\text{and}\quad
B:=\big\{l\in\Z\,\big|\,(p^\bullet_l,q^\bullet_l)=(0,0)\big\}.
\]
Note that for $l\in B$ the relations \eqref{eq:relation1} are trivial.
More generally, by taking the partial derivatives $\partial_{p_n}$ and $\partial_{q_n}$ of \eqref{eq:integrals} in 
$V\cap\big(\h^2\times\h^2\big)$ for $n\ne l$ we see that for any $l\in\Z$ and for any $n\ne l$ we have
\begin{equation}
\frac{\partial^2 H^c}{\partial p_n\partial q_l}\,p_l-
\frac{\partial^2 H^c}{\partial p_n\partial p_l}\,q_l=0
\quad\text{and}\quad
\frac{\partial^2 H^c}{\partial q_n\partial q_l}\,p_l-
\frac{\partial^2 H^c}{\partial q_n\partial p_l}\,q_l=0
\end{equation}
for any $(p,q)\in V\cap\big(\h^2\times\h^2\big)$.
This and Lemma \ref{lem:center} below, applied to $I$ equal to $\big(p_l^2+q_l^2\big)/2$ and
$F$ equal to $\frac{\partial H^c}{\partial p_n}$ and $\frac{\partial H^c}{\partial q_n}$ respectively, implies that for 
any $l\in B$ and for any $n\ne l$,
\begin{equation}\label{eq:relation2}
\frac{\partial^2 H^c}{\partial p_n\partial q_l}\Big|_{z^\bullet}=
\frac{\partial^2 H^c}{\partial p_n\partial p_l}\Big|_{z^\bullet}=0
\quad\text{and}\quad
\frac{\partial^2 H^c}{\partial q_n\partial q_l}\Big|_{z^\bullet}=
\frac{\partial^2 H^c}{\partial q_n\partial p_l}\Big|_{z^\bullet}=0.
\end{equation}
By combining \eqref{eq:relation2} with \eqref{eq:L-tilde} we obtain
\begin{eqnarray}\label{eq:L-tilde'}
\widetilde{\LL}_c&=&\sum\limits_{n\in A}
\partial_{q_n}\otimes\sum\limits_{l\in A}\Big(\frac{\partial^2 H^c}{\partial p_n\partial p_l}\,dp_l+
\frac{\partial^2 H^c}{\partial p_n\partial q_l}\,dq_l\Big)\Big|_{z^\bullet}\nonumber\\
&-&\sum\limits_{n\in A}
\partial_{p_n}\otimes\sum\limits_{l\in A}\Big(\frac{\partial^2 H^c}{\partial q_n\partial p_l}\,dp_l+
\frac{\partial^2 H^c}{\partial q_n\partial q_l}\,dq_l\Big)\Big|_{z^\bullet}\nonumber\\
&+&\sum\limits_{n\in B}
\big(\partial_{p_n},\partial_{q_n}\big)\otimes
\begin{pmatrix}
-\frac{\partial^2 H^c}{\partial q_n\partial p_n}&-\frac{\partial^2 H^c}{\partial q_n\partial q_n}\\
\frac{\partial^2 H^c}{\partial p_n\partial p_n}&\frac{\partial^2 H^c}{\partial p_n\partial q_n}
\end{pmatrix}
_{\big|_{z^\bullet}}
\begin{pmatrix}
dp_n\\
dq_n
\end{pmatrix}
.
\end{eqnarray}
Since the local Birkhoff coordinates are assumed gauge invariant and since $dH(X_{I_k})=\{H,I_k\}=0$ 
for any $k\in\Z$, we conclude that the flow $S_k^t$ of the vector field $X_{I_k}$ preserves $X_{H^c}$, that is
for any $t\in\R$ and for any $(p,q)\in V\cap\big(\h^2\times\h^2\big)$ such that 
$S_k^t(p,q)\in V\cap\big(\h^2\times\h^2\big)$ we have the following commutative diagram
\begin{equation}\label{eq:diagram_vector_field_k}
\begin{tikzcd}
V\cap\big(\h^2\times\h^2\big)\arrow{r}{X_{H^c}}\arrow[swap]{d}{S_k^t}&\ell^2\times\ell^2\arrow{d}{S_k^t}\\
V\cap\big(\h^2\times\h^2\big)\arrow{r}{X_{H^c}}&\ell^2\times\ell^2
\end{tikzcd}
.
\end{equation}
\begin{Rem}
Note that for any $s\in\R$ and for any $k\in\Z$ we have that $X_{I_k}=p_k\partial_{q_k}-q_k\partial_{p_k}$ is
a vector field in the proper sense (i.e. non-weak) on $\h^s\times\h^s$ and that 
$S_k^t : \h^s\times\h^s\to \h^s\times\h^s$ is a bounded linear map. In fact, if we introduce complex variables 
$z_k:=p_k+iq_k$, $k\in\Z$, then 
\[
\big(S_k^t(z)\big)_l=\left\{
\begin{array}{lr}
z_l,&l\ne k,\\
e^{-it} z_k,&l=k.
\end{array}
\right.
\]
\end{Rem}
\noindent In particular, we see from \eqref{eq:diagram_vector_field_k} that for any $k\in\Z$ and for any $t\in\R$ near zero 
we have that $\big(d_{S_k^t(z^\bullet)}X_{H^c}\big)\circ S_k^t=S_k^t\circ(d_{z^\bullet} X_{H^c})$.
For $k\in B$ we have $S_k^t(z^\bullet)=z^\bullet$ and hence, for any $t\in\R$ near zero,
\[
\widetilde{\LL}_c\circ S_k^t= S_k^t\circ\widetilde{\LL}_c.
\]
By taking the $t$-derivative at $t=0$ we obtain that for any $k\in B$,
\begin{equation}\label{eq:commutator_relation}
\big[\widetilde{\LL}_c,d_{z^\bullet}X_{I_k}\big]=0\quad\text{where}\quad
d_{z^\bullet}X_{I_k}=\partial_{q_k}\otimes dp_k-\partial_{p_k}\otimes dq_k.
\end{equation}
Formula \eqref{eq:L-tilde'} together with \eqref{eq:commutator_relation} and \eqref{eq:relation1} then implies that
\begin{equation}\label{eq:L-tilde_final}
\widetilde{\LL}_c=
\sum_{n,k\in A} A_{nk}\,\,\big(X_{I_n}\big|_{z^\bullet}\big)\otimes \big(d_{z^\bullet} I_k\big)
+\sum_{n\in B} B_n\,\big(\partial_{q_n}\otimes dp_n-\partial_{p_n}\otimes dq_n\big)
\end{equation}
for some matrices $\big(A_{nk})_{n,k\in A}$ and $\big(B_n\big)_{n\in B}$ with constant elements.
Note that in view of the commutative diagram \eqref{eq:diagram_operators} and Theorem \ref{th:the_spectrum_of_X_H},
the unbounded operator $\widetilde{\LL}_c$ on  $\ell^2\times\ell^2$ with domain $\h^2\times\h^2$ has a compact resolvent. 
In particular, it has discrete spectrum. Moreover, by Theorem \ref{th:the_spectrum_of_X_H} (i), zero belongs to 
the spectrum of $\widetilde{\LL}_c$ and has geometric multiplicity one. Since, in view of \eqref{eq:L-tilde_final}, 
the vectors $X_{I_k}\big|_{z^\bullet}$, $k\in A$, are eigenvectors of $\widetilde{\LL}_c$ with eigenvalue zero, 
we conclude that $A$ consists of one element $A=\{n_0\}$ and that $B_n\ne 0$ for any $n\in\Z\setminus\{n_0\}$. 
Hence, the spectrum of $\widetilde{\LL}_c$ consists of $\big\{\pm i B_n\big\}_{n\in\Z\setminus\{n_0\}}$ and zero, 
which has algebraic multiplicity two and geometric multiplicity one. This completes the proof of 
Lemma \ref{lem:the_spectrum_of_X_H_tilde}.
\end{proof}

Let $\{(x,y)\}$ be the coordinates in $\R^2$ equipped with the canonical symplectic 
form $d x\wedge dy$ and let $I=(x^2+y^2)/2$. The proof of the following Lemma is not complicated and thus omitted.

\begin{Lem}\label{lem:center}
If $F : \R^2\to\R$ is a $C^1$-map such that $\{F,I\}=0$ in some open neighborhood of 
zero then $d_{(0,0)}F=0$.
\end{Lem}

Now, we are ready to prove Theorem \ref{th:main}.

\begin{proof}[Proof of Theorem \ref{th:main}]
Take $c\in\C$ such that $|c|\notin\pi\Z$ and $|c|>\pi$, and assume that there exist gauge invariant local Birkhoff coordinates 
of the focusing NLS equation in a neighborhood of the constant potential $\varphi_c$.
In view of Lemma \ref{lem:c-real} and Theorem \ref{th:the_spectrum_of_X_H} (i) the spectrum of $\LL_c$ on $i L^2_r$ is 
discrete and contains non-zero real eigenvalues. On the other side, by Lemma \ref{lem:the_spectrum_of_X_H_tilde}, 
the spectrum of $\widetilde{\LL}_c$ lies on the imaginary axis. This shows that the two operators are not conjugated 
and hence, contradicts the existence of local Birkhoff coordinates.
\end{proof}



\begin{thebibliography}{99}   

\bibitem{Arn} V. Arnold, Mathematical methods of classical mechanics, 
Graduate Texts in Mathematics, $\bf 60$, Springer-Verlag, New York, 1989

\bibitem{GK1} B. Gr\'ebert, T. Kappeler, {\em The defocusing NLS equation and its normal form},
 EMS Series of Lectures in Mathematics, EMS, Z\"urich, 2014
 
\bibitem{KT} T. Kappeler, P. Topalov, {\em Arnold-Liouville theorem for integrable PDEs: a case study 
of the focusing NLS equation}, in preparation

\bibitem{Lax} P. Lax, Functional Analysis, Wiley, 2002
 
\bibitem{LM} Y. Li, D. McLaughlin, {\em Morse and Melʹnikov functions for NLS PDEs},
Comm. Math. Phys., $\bf 162$(1994), no. 1, 175-214
 
\bibitem{Wi} J. Williamson, {\em On the algebraic problem concerning the normal forms of linear
dynamical systems}, Amer. J. Math., $\bf 58$(1936), 141-163

\bibitem{Yoshida} K. Yoshida, {\em Functional Analysis}, 3rd ed., Springer-Verlag, New York, 1971

\bibitem{Zung} N. Zung,  {\em A note on focus-focus singularities}, Differential Geometry and its Applications,
$\bf 7$(1997), 123-130

\end{thebibliography}
\end{document}